\documentclass[a4,11pt]{article}
\usepackage{mathrsfs}
\usepackage{graphicx,graphics,epsfig}
\usepackage{amsmath,amsthm,amssymb}
\newtheorem{Def}{Definition}
\newtheorem{Prop}{Proposition}
\newtheorem{lem}{Lemma}
\newtheorem{thm}{Theorem}
\newtheorem{cor}{Corollary}
\newtheorem{ass}{Assumption}
\newtheorem{rem}{Remark}
\newcommand{\ds}{\displaystyle}

\newcommand{\R}{\mathbb{R}}
\newcommand{\N}{\mathbb{N}}

\newcommand{\disp}{\displaystyle}
\begin{document}
\date{}
\title{The minimum value function for the Tikhonov regularization and its applications.}

\author{
   K.Ito\thanks{
  Department of Mathematics, North Carolina State
  University, Raleigh, NC.
  ({\tt kito@math.ncsu.edu })
}
   and
   T.Takeuchi\thanks{
  Center for Research in Scientific Computation, North Carolina State University, Raleigh, NC.
  ({\tt tntakeuc@ncsu.edu})}
}%

\maketitle

\begin{abstract}
The minimum value function appearing in Tikhonov regularization technique is very useful in determining the regularization parameter, both theoretically and numerically. In this paper, we discuss the properties of the minimum value function. We also propose an efficient method to determine the regularization parameter. A new criterion for the determination of the regularization parameter is also discussed.
\end{abstract}

\section{Introduction}
Let consider the generic model of inverse problems
$Kx=y$ where $x\in X$ and $y\in Y$ refer to the unknown and the data in a Banach space $X$ and $Y$, and $K$ is a nonlinear operator.
The model can represent a variety of inverse problems arising
in industrial applications including
computerized tomography, inverse scattering and image processing.

Due to the ill-poisedness of the problem, a regularization method must be applied
in order to retrieve $x$ from the noisy data $y^\delta$ and there are numerous works devoted for regularization methods.
One of the most appealing regularization techniques is
Tikhonov regularization, which has been studied from both theoretical and computational aspects by many authors.

The Tikhonov regularization takes the minimizer of $x_\alpha$
to the functional $J_\alpha$ in an admissible set $Q_{\text{ad}}\subset X$:
\begin{equation}\label{TikhonovFunctional}
x_\alpha:=\arg\inf_{x\in Q_{\text{ad}}}J_\alpha(x),
\end{equation}
where
\[
J_\alpha(x)=\varphi(x,y^\delta)+\alpha\psi(x).
\]
Here $\alpha>0$ is the regularization parameter compromising the data
fitting and fidelity term
$\varphi$ and a priori information encoded in the restoration energy functional
$\psi$ of $x$. 
Commonly used data fidelity functionals $\varphi $ include
$\|Kx-y^\delta\|^2_{L^2}$, $\|Kx-y^\delta\|_{L^1}$ and
$\int (Kx-y^\delta \log Kx)$ and regularization functionals
$\psi$ include $(Bx,x)_X$ with a bounded linear self-adjoint nonnegative operator $B$, $\|x\|_{TV}$, etc.
The set $Q_{\text{ad}}$ which describes \textit{a priori} information for the solution $x_0$ is usually set to be weakly closed convex set in $X$. For instance, $Q_{\text{ad}}=Z\cap\{x\in Z \mid \|x\|_{Z}\le M\}$ where $Z\subset X$ is the another Banach space whose topology is stronger than that of $X$ and $M>0$ is a given constant. For the detail of Tikhonov regularization,
we refer to Baumeister \cite{Baumeister-Stabsoluinveprob:87}, Engl,
Hanke and Neubauer \cite{Engl+HankeETAL-Reguinveprob:96},
Groetsch
\cite{Groetsch-Inveprobmathscie:93},
Hofmann
\cite{Hofmann-Reguapplinveill-:86}
  and references therein.

The selection of a parameter $\alpha$ is crucial for the stable solution and there is a significant amount of works for the development of methods for choosing a suitable parameter. Among others, we refer to \cite{Cheng+Yamamoto-strapriochoiregu:00,Engl+HankeETAL-Reguinveprob:96,Groetsch-theoTikhreguFred:84}. 

The minimum value function (see Definition \ref{Def_value}) appearing in Tikhonov regularization technique is very useful in determining the regularization parameter. In \cite{Tikhonov+LeonovETAL-Nonlill-prob:98} on several principles 
such as the generalized principle of discrepancy, the generalized principle of quasisolutions, the generalized principle of smoothing functional, and the principles are investigated by the calculus of the minimum value function.

The value function calculus also gives an insight to well known conventional principles such as Morozov principle. For example, Ito and Kunisch proposed in \cite{Ito+Kunisch-choireguparanonl:92} study the Morozov principle in terms of the minimum value function for nonlinear inverse problems. On the basis of the value function calculus, they propose a modified Morozov principle. As we see in section 4, other conventional principles can also be formulated in terms of the value function.

In all the principles we have mentioned, each of the regularization parameters is determined with an equation including the value function and its higher derivatives. It can be computationally expensive to solve the equation numerically. In order to reduce the computational effort, an efficient method and an algorithm should be developed.

A model function approach is proposed in \cite{Ito+Kunisch-choireguparanonl:92}, in which an approximation (a model function) to the minimum value function is constructed and use the model function for the value function in their principle. Other principles and numerical algorithms can be found in
Kunisch and Zou \cite{Kunisch+Zou-Iterchoiregupara:98}, Xie and Zou \cite{Xie+Zou-imprmodefuncmeth:02}.

In this paper, new results on the properties of the minimum value function are derived in a general set-up for problem (\ref{TikhonovFunctional}), and a new principle for a choice of a regularization parameter is proposed using the minimum value function, which strongly relates to the principle of Reginska \cite{Reginska-reguparadiscill-:96} also known as the minimum product criterion.
We also propose a model function for the value function and employ the model function approach in several conventional principles to numerically compute the regularization parameters accurately and efficiently.

The paper is organized as follows. In section 2, we give properties of the minimum value function.
A model function is proposed in section 3 and the efficiency of the model function is verified in section 4. In section 5, a new principle for the regularization parameter is given.
\section{The minimum value function and its properties}\label{sec:F}
In this section, we investigate the properties of the minimum value function in a general framework :
Consider the generic model of inverse problems
$Kx=y$ where $K$ is a nonlinear operator from a Banach space $X$ and another Banach space $Y$. We retrieve $x$ from the noisy data $y^\delta$ using Tikhonov regularization technique, i.e., $x$ is approximately obtained with the minimizer $x_\alpha$
\begin{equation}
x_\alpha:=\arg\inf_{x\in Q_{\text{ad}}}J_\alpha(x).
\end{equation}
\begin{Def}\label{Def_value} The minimal value function $F(\alpha)$ of $J_\alpha$ is the function of the parameter $\alpha$ defined as
\begin{equation}\label{valuef}
F(\alpha)=\inf_{x\in Q_{\text{ad}}}J_\alpha(x).
\end{equation}
\end{Def}
\begin{lem}\label{lem:F}
The value minimum function $F(\alpha)$ is (i) monotonically increasing and (ii) concave.
\end{lem}
\begin{proof}
(i): Let $0<\hat{\alpha}<{\alpha}$ be given. For any $x\in Q_{\text{ad}}$,
\[
F(\hat{\alpha})\le J_{\hat{\alpha}}(x) =\varphi(x,y^\delta)+\hat{\alpha}\,\psi(x)
\le \varphi(x,y^\delta)+{\alpha}\,\psi(x).
\]
Taking the infimum with respect to $x$ yields $F(\hat{\alpha})\le F({\alpha})$.\\
(ii): Let $\alpha_1>0$ and $\alpha_2>0$ be given. Set $\gamma = (1-t)\alpha_1+t\alpha_2$ for $0\le t\le 1$, then
\begin{eqnarray*}
F((1-t)\alpha_1+t\alpha_2) &=& \inf_{x}J_{\gamma}(x)
=\inf_{x}\left(\varphi(x,y^\delta)+ ((1-t)\alpha_1+t\alpha_2)\psi(x)\right) \\
&\ge & (1-t)\inf_x(\varphi(x,y^\delta)+\alpha_1 \psi(x)) +
t\inf_x(\varphi(x,y^\delta)+\alpha_2 \psi(x)) \\
& = & (1-t)F({\alpha_1}) + t F({\alpha_2}).
\end{eqnarray*}
Hence $F(\alpha)$ is concave.
\end{proof}
Since $F(\alpha)$ is concave, it is continuous.
\begin{cor}\label{cor:DFae}
$F(\alpha)$ is continuous everywhere.
\end{cor}
\begin{rem}
Lemma \ref{lem:F} does not require the existence of $x\in Q_{\text{ad}}$ that achieves the infimum of $J_\alpha$.
\end{rem}
We examine the minimum value function more closely. Let $D^{+}F$ and $D^{-}F$ are one-sided derivatives of the value function $F$, i.e.,
\[
D^{+}F(\alpha) = \lim_{h \downarrow 0} \frac{F(\alpha)-F(\alpha-h)}{h},\quad
D^{-}F(\alpha) = \lim_{h \downarrow 0} \frac{F(\alpha+h)-F(\alpha)}{h}.
\]

Note that both limits exist for all $\alpha>0$ and that $D^{\pm}F(\alpha)\ge 0$. Indeed,
for given $\alpha>0$, let $0<h_1<h_2<\alpha$. Select $t$ as $t=1-\frac{h_1}{h_2}<1$,
then $\alpha-h_1= t\alpha+(1-t)(\alpha-h_2)$. By the convexity of $F(\alpha)$,
\[
F(\alpha-h_1)\ge tF(\alpha)+(1-t)F(\alpha-h_2) = \left(1-\frac{h_1}{h_2}\right)F(\alpha)+\frac{h_1}{h_2}F(\alpha-h_2).
\]
Then
\begin{eqnarray*}
  &&h_2\left(\frac{F(\alpha)-F(\alpha-h_2)}{h_2}-\frac{F(\alpha)-F(\alpha-h_1)}{h_1}\right)
  =F(\alpha)-F(\alpha-h_2)-\frac{h_2}{h_1}F(\alpha)+\frac{h_2}{h_1}F(\alpha-h_1)\\
  \quad && \ge\left(1-\frac{h_2}{h_1}\right)F(\alpha)-F(\alpha-h_2)+
  \frac{h_2}{h_1}\left\{\left(1-\frac{h_1}{h_2}\right)F(\alpha)+\frac{h_1}{h_2}F(\alpha-h_2) \right\}=0.
\end{eqnarray*}
Therefore, the limit $\displaystyle \lim_{h\downarrow 0}\frac{F(\alpha)-F(\alpha-h)}{h}\ge 0$ exists.

Here we list two the basic properties of $D^{\pm}F(\alpha)$.
\begin{enumerate}
\item{\bf Monotonicity}:  $0\le D^{-}F({\alpha})\le D^{+}F(\alpha)\le D^{-}F(\hat{\alpha})\le D^{+}F(\hat{\alpha})$, for all $0<\hat{\alpha}<{\alpha}$,
\item{\bf Left and Right continuity}: $\displaystyle D^{+}F(\alpha) = \lim_{h \downarrow 0} D^{+}F(\alpha-h)$, and
$\displaystyle D^{-}F(\alpha) = \lim_{h \downarrow 0} D^{-}F(\alpha+h)$ for all $\alpha>0$.
\end{enumerate}
The differentiability of $F$ at $\alpha$ guarantees the continuity of $D^{\pm}F$ at this point. Indeed, the monotonicity of $D^{\pm}F$ and the left continuity of $D^-F$ yield the inequalities $\disp D^-F(\alpha) = \lim_{h\downarrow 0}D^-F(\alpha+h)\le \lim_{h\downarrow 0}D^+F(\alpha+h) \le D^+F(\alpha)$. Now suppose $F$ is differentiable at $\alpha$, i.e, $D^{+}F(\alpha)=D^{-}F(\alpha)$. Then from the inequalities it follows that $ \disp\lim_{h\downarrow 0}D^+F(\alpha+h) = D^+F(\alpha)$, which shows the continuity of $D^{+}F$ at $\alpha$. Similarly it follows that $D^-F$ is continuous at $\alpha$. The fact is used in section \ref{sec:DF}.

For the analysis of the minimum value function, we introduce the following function of $\gamma>0$:
\begin{eqnarray*}
G(\gamma)&:=&\gamma F(\frac{1}{\gamma})=\gamma \inf_{x\in Q_{\text{ad}}} \{\varphi(x,y^\delta)+\frac{1}{\gamma}\psi(x)\}\\
&=&\inf_{x\in Q_{\text{ad}}} \{\psi(x)+\gamma\,\varphi(x,y^\delta)\}.
\end{eqnarray*}
Note that once a result on $F(\alpha)$ is proved then the same result is true for $G(\gamma)$ as well. For example, we can show that $G(\gamma)$ is concave in exactly the same way as Lemma \ref{lem:F}. Every results on $G(\gamma)$ can be written in terms of $F(\alpha)$ by using the relation:
\begin{equation}\label{equ:DG}
D^{\pm}G(\gamma)|_{\gamma=\alpha^{-1}} = F(\alpha)-\alpha D^{\mp}F(\alpha),
\end{equation}
which can be verified without difficulty.
\subsection{Asymptotic property of $F(\alpha)$}
We study the asymptotic behavior of $F(\alpha)$.
\begin{Prop}\label{prop:asymp}
\begin{equation}
  \lim_{\alpha\downarrow 0}F(\alpha)
  = \inf_{x\in Q_{\text{ad}}}\varphi(x,y^\delta), \quad \lim_{\alpha\downarrow 0}\alpha D^{\pm}F(\alpha)=0.
\end{equation}
\end{Prop}
\begin{proof}
Let $C:=\inf_{x\in Q_{\text{ad}}}\varphi(x,y^\delta)$. For any $\varepsilon>0$ there exists $x^\varepsilon \in Q_{\text{ad}}$ such that
$C\le\varphi(x^\varepsilon,y^\delta) < C +\varepsilon$. Then
\begin{eqnarray*}
  C&\le& \inf_{x\in  Q_{\text{ad}}}\varphi(x,y^\delta)+\alpha\,\inf_{x\in  Q_{\text{ad}}}\psi(x)\le F(\alpha)
  \le J_\alpha(x^\varepsilon) \\
  &=& \varphi(x^\varepsilon,y^\delta)+\alpha\,\psi(x^\varepsilon)< C + \varepsilon + \alpha\,\psi(x^\varepsilon).
\end{eqnarray*}
By passing to the limit $\alpha\downarrow 0$ and taking into account that $\varepsilon$ is arbitrary, we obtain $\displaystyle C=\lim_{\beta
\downarrow 0}F(\alpha)$.

On the other hand, as we see above we have $\displaystyle D^{+}F(\alpha)\le \frac{F(\alpha)-F(\alpha-h)}{h}$ for all $h<\alpha$, here $\alpha$ is arbitrary fixed. Since $\displaystyle\lim_{h\uparrow\alpha}F(\beta
-h)= C$, we obtain
\[ D^{+}F(\alpha)\le \lim_{h\uparrow\alpha}\frac{F(\alpha)-F(\alpha-h)}{h} = \frac{F(\alpha)-C}{\alpha},
\]
which means $\displaystyle C\le F(\alpha)-\alpha D^{+}F(\alpha)$. Therefore $C\le F(\alpha)-\alpha D^{+}F(\alpha)\le F(\alpha)$. From this inequality we get $\displaystyle\lim_{\alpha\downarrow 0}\alpha D^{+}F(\alpha)=0$.

Let $\{\alpha_j\}$ be a sequence that converges to 0 as $j\to \infty$. We pick $\hat{\alpha}_j$ such that
$\hat{\alpha}_j<\alpha_j<2\hat{\alpha}_j$ to get the sequence $\{\hat{\alpha}_j\}$ such that $\hat{\alpha}_j\to 0$. By monotonicity of $D^{\pm}F(\alpha)$ we have
\begin{eqnarray*}
0&\le&
\alpha_j D^{-}F(\alpha_j)\le {\alpha}_jD^{+}F(\hat{\alpha}_j)\\
&=&\hat{\alpha}_jD^{+}F(\hat{\alpha}_j)+(\alpha_j-\hat{\alpha}_j)D^{+}F(\hat{\alpha}_j)
<2\hat{\alpha}_jD^{+}F(\hat{\alpha}_j).
\end{eqnarray*}
Passing to the limit $j\to\infty$ yields $\displaystyle \lim_{\alpha\downarrow 0}\alpha D^{-}F(\alpha)=0$.
\end{proof}
From this proposition it follows that $F(\alpha)$ is right continuous at $0$.
\begin{Prop}
\begin{equation}
  \lim_{\alpha\rightarrow \infty}\frac{F(\alpha)}{\alpha}
  =\lim_{\alpha\rightarrow \infty} D^{\pm}F(\alpha)
  = \inf_{x\in Q_{\text{ad}}}\psi(x).
\end{equation}
\end{Prop}
\begin{proof}
Applying Proposition \ref{prop:asymp} to $G(\gamma)$ yields
\begin{equation*}
  \lim_{\gamma\downarrow 0}G(\gamma)
  =\inf_{x\in Q_{\text{ad}}}\psi(x),\quad \lim_{\gamma\downarrow 0}\gamma D^{\pm}G(\gamma)=0,
\end{equation*}
or equivalently,
\begin{equation*}
  \lim_{\alpha\to \infty}G(\alpha^{-1})=\inf_{x\in Q_{\text{ad}}}\psi(x),\quad
  \lim_{\alpha\to \infty}\alpha^{-1} D^{\pm}G(\alpha^{-1})=0.
\end{equation*}
By using the relation (\ref{equ:DG}), we arrive at the desired results.
\end{proof}
\subsection{Differentiability of $F$}\label{sec:DF}
We investigate the relation between the derivative $F'$ and the two terms, the fidelity functional $\psi$ and the regularization functional $\varphi$. We also show a sufficient condition for the existence of $F'$ in terms of $\varphi$ and $\psi$.
\begin{Def}
We denote the set of solutions of the minimization problem
$\displaystyle \inf_{x\in Q_{\text{ad}}}J_\alpha(x)$ by $\mathcal{M}_\alpha$.
\end{Def}
For simplicity of our argument we always assume the existence of the minimizer.

It may happen that there exist two minimizers $x_\alpha \neq \hat{x}_\alpha \in \mathcal{M}_\alpha$ satisfying
\[
F(\alpha) = \varphi(x_\alpha,y^\delta)+\alpha\psi(x_\alpha) = \varphi(\hat{x}_\alpha,y^\delta)+\alpha\psi(\hat{x}_\alpha)
\]
and
\[
\varphi(x_\alpha,y^\delta)<\varphi(\hat{x}_\alpha,y^\delta) \mbox{ and }  \psi(x_\alpha)>\psi(\hat{x}_\alpha).
\]
That is, for a fixed $\alpha$, the value $\varphi(x_\alpha,y^\delta)$ and $\psi(x_\alpha)$ may vary depending on the choice of the minimizer
$x_\alpha\in\mathcal{M}_{\alpha}$. Thus, it is possible that the maps $\alpha\rightarrow \varphi(x_\alpha,y^\delta)$ and $\alpha\rightarrow \psi(x_\alpha)$ will be multi-value functions.
In what follows, we study the basic properties of those functions.
We begin with the following inequality.

\begin{lem}\label{lem:DF}
The following inequalities hold for all $x_\alpha\in\mathcal{M}_\alpha$:
\begin{align}\label{ineq:DFandpsi}
D^{-}F(\alpha)\le \ &\psi(x_\alpha) \le D^{+}F(\alpha),\\
F(\alpha)-\alpha\,D^{+}F(\alpha)\le \ &\varphi(x_\alpha,y^\delta) \le F(\alpha)-\alpha\, D^{-}F(\alpha).
\label{ineq:DFandphi}
\end{align}
\end{lem}
\begin{proof}
For arbitrary $\hat{\alpha}$ such that $0<\hat{\alpha}<\alpha$,
\[
F(\hat{\alpha}) = \inf_{x\in Q_{\text{ad}}} J_{\hat{\alpha}}(x)\le J_{\hat{\alpha}}(x_\alpha)
= \varphi(x_\alpha,y^\delta) + \hat{\alpha}\psi(x_\alpha).
\]
Thus,
\begin{align*}
F(\alpha)-F(\hat{\alpha})
 &\ge\varphi(x_\alpha,y^\delta) + \alpha\psi(x_\alpha)-\varphi(x_{\alpha},y^\delta)
-\hat{\alpha}\psi(x_\alpha)\\
& = (\alpha-\hat{\alpha})\psi(x_\alpha).
\end{align*}
Thus we obtain
\begin{equation*} \label{(1)}
\frac{F(\alpha)-F(\hat{\alpha})}{\alpha-\hat{\alpha}} \ge \psi(x_{\alpha}).
\end{equation*}
By passing to the limit $\hat{\alpha}\uparrow \alpha$, it follows that $D^{+}F(\alpha)\ge \psi(x_\alpha)$.
Similarly, we obtain $D^{-}F(\alpha)\le \psi(x_\alpha)$ and thus the inequality (\ref{ineq:DFandpsi}) is proven.

We also obtain the second inequality (\ref{ineq:DFandphi}) from (\ref{ineq:DFandpsi}) and the definition of $F(\alpha)$.
\end{proof}
\begin{cor}\label{cor:DF=psi}
If $F'(\alpha)$ exists at $\alpha$, then $\psi(x_{\alpha})$ and $\varphi(x_{\alpha},y^\delta)$ are single valued at $\alpha$ and it holds that
$F'(\alpha)=\psi(x_{\alpha})$ and $F(\alpha)-\alpha\,F'(\alpha)=\varphi(x_{\alpha},y^\delta)$  for all $x_{\alpha}\in\mathcal{M}_{\alpha}$.
\end{cor}
Note a monotone increasing (decreasing) function is differentiable except on a possibly countable set.
\begin{cor}\label{cor:DF=psi}
There exists a possibly countable set $N$ such that,
\begin{itemize}
  \item $F$ is differentiable and the multi-value functions $\varphi(x_\alpha,y^\delta)$ and $\psi(x_{\alpha})$
have single value on $\alpha \in N^{\complement}$.
  \item $\varphi(x_{\alpha},y^\delta)=F(\alpha)-\alpha\,F'(\alpha)$, $\psi(x_{\alpha})=F'(\alpha)$ for all $x_{\alpha}\in\mathcal{M}_{\alpha}$ if $\alpha \in N^{\complement}$.
\end{itemize}
\end{cor}
Corollary \ref{cor:DF=psi} guarantees the differentiability of $F(\alpha)$ except on a possibly countable set.

Next we show the conditions for the differentiability of $F(\alpha)$ at all $\alpha>0$. Firstly, we define the $\psi$ boundness to state the assumptions.
\begin{Def}
A sequence $\{x_n\}_{n=1}^\infty\subset Q_{ad}$ is $\psi$-bounded if there exists constant $M>0$ such that
 $\sup_n\psi(x_n)<M$ for all $n\in\N$.
\end{Def}
\begin{ass}\label{ass:2}
\begin{itemize}
\item Let $\{x_n\}_n\subset Q_{ad}$ be a $\psi$-bounded sequence. There exists a subsequence $\{x_{n_k}\}_k$ which weakly converges to an element $x^*\in Q_{ad}$ in the topology of $X$.
\item $\varphi$ and $\psi$ are lower semi-continuous with respect to weakly convergence sequencers, i.e, if
a subsequence $\{x_{n}\}_n$ which weakly converges to an element $x^*\in Q_{ad}$, then
\[
\varphi(x^*)\le\liminf_{n \to \infty}\varphi(x_{n}) \mbox{ and }\,
\psi(x^*)\le\liminf_{n \to \infty }\psi(x_{n}).
\]
\end{itemize}
\end{ass}
Henceforth hereafter we assume the assumption holds.
Then is guaranteed the existence of the solutions $x^{\pm}_{\alpha}$ in $\mathcal{M}_\alpha$ such that $\psi(x^{\pm}_{\alpha})=D^{\pm}F(\alpha)$.
\begin{thm} There exist $x^{+}_{\alpha}$ and $x^{-}_{\alpha}$ in $\mathcal{M}_\alpha$ such that
$\psi(x^{+}_{\alpha})=D^{+}F(\alpha)$ and $\psi(x^{-}_{\alpha})=D^{-}F(\alpha)$ for all $\alpha>0$.
\end{thm}
\begin{proof}
Let $\alpha>0$ fixed arbitrary and let $h>0$ be a parameter such that $h<<\alpha$ and $h\rightarrow 0$. Let $\{x_{\alpha-h}\}_{h}$ be a minimizing sequence. Then from the monotonicity of $F(\alpha)$,
\[
J_{\alpha-h}(x_{\alpha-h})=F(\alpha-h) \le F(\alpha) \quad \mbox{ for all } h>0.
\]
Thus $\displaystyle \psi(x_{\alpha-h})<\frac{F(\alpha)}{\alpha-h}$ and it follows that the sequence $\{x_{\alpha-h}\}_{h}$ is bounded. By Assumption \ref{ass:2}, there exists subsequence of $\{x_{\alpha-h}\}_h$, which we denote it by $\{x_{\alpha-h}\}_h$, that converges weakly to an element $x^*\in Q_{\text{ad}}$. Then by the continuity of $F(\alpha)$ and the lower semi-continuity of $\varphi$ and $\psi$, it follows that
\begin{eqnarray*}
F(\alpha)&=&\lim_{h\downarrow 0}F(\alpha-h)
\ge\varliminf_{h\downarrow 0}\varphi(x_{\alpha-h},y^\delta) + \alpha\varliminf_{h\downarrow 0}\psi(x_{\alpha-h})\\
&\ge& \varphi(x^*,y^\delta) + \alpha\varliminf_{h\downarrow 0}\psi(x_{\alpha-h})\\
&\ge& \varphi(x^*,y^\delta) + \alpha\psi(x^*)=J_\alpha(x^*)\ge F(\alpha),
\end{eqnarray*}
and thus we have $\displaystyle \psi(x^*)=\varliminf_{h\downarrow 0}\psi(x_{\alpha-h})$ and $x^*\in\mathcal{M}_\alpha$. In what follows, we show that 
$\psi(x^*)=D^{+}F(\alpha)$.

By Lemma \ref{lem:DF},
\[
\varliminf_{h\downarrow 0}D^{-}F(\alpha-h)\le \varliminf_{h\downarrow 0}\psi(x_{\alpha-h})
\le \varliminf_{h\downarrow 0}D^{+}F(\alpha-h) = D^{+}F(\alpha).
\]
The last equality follows from the left continuity of $D^{+}F(\alpha)$. Since $\displaystyle \varliminf_{h\downarrow 0}D^{-}F(\alpha-h)=D^{+}F(\alpha)$,
we obtain $\displaystyle \varliminf_{h\downarrow 0}\psi(x_{\alpha-h})=D^{+}F(\alpha)$ and therefore $\psi(x^*)=D^{+}F(\alpha)$.

Similarly we can show the existence of the minimizer $x^-_\alpha \in \mathcal{M}_{\alpha}$ that satisfies $\psi(x^-_\alpha)=D^{-}F(\alpha)$. We complete the proof.
\end{proof}
\begin{cor}
There exist elements $x^{+}_\alpha$ and $x^{-}_\alpha$ such that $\displaystyle \psi(x^{+}_\alpha) = \max_{x \in \mathcal{M}_\alpha}\psi(x)$ and
$\displaystyle \psi(x^{-}_\alpha) = \min_{x \in \mathcal{M}_\alpha}\psi(x)$.
\end{cor}
\begin{cor}\label{cor:22}
If $\psi(x_\alpha)=\psi(\hat{x}_\alpha)$ for all $x_\alpha, \hat{x}_\alpha \in \mathcal{M}_\alpha$ for all $\alpha>0$,
 then $F'(\alpha)$ exists and it is continuous 
 for all $\alpha>0$.
\end{cor}
\begin{cor}\label{cor:2}
Assume that the solution of minimization problem (\ref{TikhonovFunctional}) is unique for all $\alpha>0$, then $F'(\alpha)$ exists and it is continuous for all $\alpha>0$.
\end{cor}
The other properties for $F(\alpha)$ such as second differentiability is studied in \cite{Ito+Kunisch-choireguparanonl:92}.

 As shown in Corollaries \ref{cor:22} and \ref{cor:2}, both of the value $F(\alpha)$ and $F^{(1)}(\alpha)$ is obtained with the knowledge of $x_\alpha$, and the computation of $\displaystyle \frac{d}{d \alpha}x_\alpha$ is not required. Moreover, we obtain $F^{(2k)}(\alpha)$ and $F^{2k+1}(\alpha)$ from $\displaystyle \frac{d^k}{d \alpha^k}x_\alpha$ for $k\ge1$ provided that $K$ is linear from Hilbert space $X$ to another Hilbert space $Y$, $\varphi(x,y^\delta) = \|Kx-y^\delta\|_Y^2$ and $\psi(x) = (Bx,x)_X$ with a symmetric operator $B$ such that (\ref{TikhonovFunctional}) has a unique solution, and $X=Q_{ad}$, i.e, no constraint is imposed.
\begin{thm}\label{thm:F2k}
The function $F(\alpha)$ is infinitely differentiable at every $\alpha>0$. The derivatives $F^{(2k)}(\alpha)$ and $F^{(2k+1)}(\alpha)$ for each $k\ge 1$ are give with the $k$-th derivative $x^{(k)}_\alpha$ as
\begin{align*}
F^{(2k)}(\alpha) &= C_k ( (K^*K+\alpha B)x^{(k)}_\alpha,x^{(k)}_\alpha)_X \\
F^{(2k+1)}(\alpha) &= -C_k(1+2k) (Bx^{(k)}_\alpha,x^{(k)}_\alpha)_X,
\end{align*}
where the constants $C_k$ are recursively defined as $\displaystyle C_{k+1} = \frac{2(2k+1)}{k+1}C_k$ with $C_1=-2$.
\end{thm}
The proof is based on the following lemma in \cite{Kunisch+Zou-Iterchoiregupara:98}.
\begin{lem}[\cite{Kunisch+Zou-Iterchoiregupara:98}]\label{lem:zou}
The function $x_\alpha$ is infinitely differentiable at every $\alpha>0$ and
its derivative $x_\alpha^{(k)}$, for each $k\ge 1$, is the unique solution to the following equation:
\[
(K^*K+\alpha B)x_\alpha^{(k)}= -kB x_\alpha^{(k-1)}
\]
\end{lem}
\begin{proof}[Proof of Theorem \ref{thm:F2k}]
From Lemma \ref{lem:zou} with $k=1$, one obtains $F^{(2)}=2(Bx_\alpha^{(1)},x_\alpha)_X=-2((K^*K+\alpha B)x_\alpha^{(1)},
x_\alpha^{(1)})_X$.

Suppose $F^{(2k)}(\alpha) = C_k((K^*K+\alpha B)x_\alpha^{(k)},x_\alpha^{(k)})_X $ with a constant $C_k$, then Lemma \ref{lem:zou} yields
\begin{align*}
F^{(2k+1)}(\alpha)&=2C_k((K^*K+\alpha B)x_\alpha^{(k+1)},x_\alpha^{(k)})_X +
C_k(Bx_\alpha^{(k)},x_\alpha^{(k)})_X\\
&=2C_k(-(k+1)Bx_\alpha^{(k)},x_\alpha^{(k)})_X +
C_k(Bx_\alpha^{(k)},x_\alpha^{(k)})_X\\
&=-C_k(1+2k)(Bx_\alpha^{(k)},x_\alpha^{(k)})_X.
\end{align*}
Then we have
\begin{align*}
F^{(2k+2)}(\alpha) &= -2C_k(1+2k)(Bx_\alpha^{(k+1)},x_\alpha^{(k)})_X\\
&= \frac{2C_k(1+2k)}{k+1}(x_\alpha^{(k+1)},-(k+1)Bx_\alpha^{(k)})_X\\
&=C_{k+1}(x_\alpha^{(k+1)},(K^*K+\alpha B)x_\alpha^{(k+1)})_X.
\end{align*}
where $C_{k+1}$ is defined as $ \displaystyle C_{k+1}=\frac{2C_k(2k+1)}{k+1}$. By induction the assertion is valid.
\end{proof}

\section{Pad\'e approximations as model functions for linear inverse problems.}\label{pade}
In this section, we propose a model function for the value function for linear inverse problems. Firstly, we give a motivation for using the model function approach.

A principle for determining a regularization parameter often requires solving an equation, for example, Morozov discrepancy principle takes the parameter $\alpha^*$ that satisfies the equation $\|Kx-y^\delta\|_X=\delta$ with noisy data $y^\delta$ of noise level $\delta$.
One can apply a Newton type iteration to solve the equation, however, the iteration could be numerically expensive. One strategy to reduce the computational effort is that: First, represent the equation in terms of the value function as  $F(\alpha)-\alpha F'(\alpha)=0$. Then, construct a model function $m(\alpha)$ to $F(\alpha)$ and find the parameter $\alpha_m$ that satisfies the equation $m(\alpha_m)-\alpha_m m'(\alpha_m)=0$. If the model function approximates to the value function $F$, the parameter $\alpha_m$ thus determined will give a close approximation to $\alpha^*$.

We assume that $K$ is a linear operator form Hilbert space $X$ and a Hilbert space $Y$, $Q_{ad}=X$ (i.e., no constraint imposed), $\varphi(x,y^\delta) = \|Kx-y^\delta\|_{Y}^2$ and $\psi(x) = (Bx,x)_{X}$ with a symmetric operator $B$ such that (\ref{TikhonovFunctional}) has a unique solution. The solution is written as
\[
x_\alpha = (K^*K+\alpha B)^{-1}K^* y^\delta.
\]
As we see in Theorem \ref{thm:F2k}, the minimum value function $F(\alpha)$ is infinitely differentiable and thus it is reasonable to consider a rational function as a model function, which is briefly mentioned in \cite{Vasileva-Ratifuncsoluline:99}. We propose our model function to $F(\alpha)$ of the particular form
\[
m(\alpha)=\|y^\delta\|_Y^2\frac{P(\alpha)}{Q(\alpha)}=\|y^\delta\|_Y^2\frac{\alpha^n + p_{n-1}\alpha^{n-1}+
\cdots p_1\alpha +p_0}{\alpha^n + q_{n-1}\alpha^{n-1}+
\cdots q_1\alpha +q_0}.
\]
The derivation of our proposed model function to $F(\alpha)$ bases on the following discussion:
Just for simplicity we assume that $X=Y=R^\ell$ and $B=I$, although our discussion is valid for the infinite dimension framework. The singular value decomposition of $K$ yields
\[
K=U\Sigma V^T,
      \quad \Sigma = \text{diag}(\sigma_1,\ldots,\sigma_\ell),
\]
where $\sigma_1\ge\cdots \ge \sigma_r > \sigma_{r+1}=\cdots =\sigma_\ell = 0$ ($r=\mbox{rank}(K))$ are the singular values and $U=[u_1,\ldots, u_\ell]$ and $V=[v_1,\ldots,v_\ell]$ are the orthogonal matrices, respectively. Then $F(\alpha)$ is represented as
\begin{equation}\label{svdF}
F(\alpha) = \sum_{k=1}^{\ell}\frac{\alpha}{\sigma_k^2+\alpha} (u_k^T y^\delta)^2
=\Vert y^\delta\Vert_{\R^\ell}^2-\sum_{k=1}^{r}\frac{\sigma_k^2}{\sigma_k^2+\alpha} (u_k^T y^\delta)^2.
\end{equation}
Since we assume that $A$ is highly ill-conditioned, the singular value $\sigma_k$ decreases rapidly as $k$ increases.
As a result only the first few $n$-terms satisfying $n<<r$ in (\ref{svdF}) will contribute to the sum. Thus we drop off the remaining terms and obtain the approximation
\begin{equation}\label{svdF2}
F(\alpha)\approx {\Vert y^\delta\Vert_{\R^\ell}^2} \left(1-\sum_{k=1}^{n}\frac{\sigma_k^2}{\sigma_k^2+\alpha}\frac{c_k^2}{\Vert y^\delta\Vert_{\R^\ell}^2}\right)
=\Vert y^\delta\Vert_{\R^\ell}^2 \frac{\alpha^n + p_{n-1}\alpha^{n-1}+
\cdots p_1\alpha +p_0}{\alpha^n + q_{n-1}\alpha^{n-1}+
\cdots q_1\alpha +q_0}.
\end{equation}

The Pad\'e approximation to $F(\alpha)$ is constructed through the use of several minimizing elements $x_\alpha$ for different values of regularization parameter $\alpha$.

For a given interval $I$, let $\{\alpha_1,\alpha_2,...,\alpha_n\}\subset I$ are $n$ distinct parameters, which we call reference points in the following. We compute the function values and its derivatives at the reference points to determine $2n$ unknowns $p_0,\ldots,p_{n-1}, q_0,\ldots q_{n-1}$ in $m(\alpha)$ by the linear system
 \begin{equation}\label{sym:F}
 m(\alpha_i)=F(\alpha_i),\qquad m^{(1)}(\alpha_i)=F^{(1)}(\alpha_i), \qquad  i=1,2,\ldots,n.
\end{equation}
The more accurate model function will be obtained by imposing further conditions on the higher derivatives to $m(\alpha)$, i.e., we determine $4n$ unknowns $p_0,\ldots,p_{2n-1}, q_0,\ldots q_{2n-1}$ by the system
 \begin{align}\label{sym:F_more}
 m(\alpha_i)&=F(\alpha_i), \quad m^{(1)}(\alpha_i) =F^{(1)}(\alpha_i),\nonumber\\
  m^{(2)}(\alpha_i)&=F^{(2)}(\alpha_i), \quad m^{(3)}(\alpha_i) =F^{(3)}(\alpha_i), \qquad  i=1,2,\ldots,n.
\end{align}
For the solvability of the systems (\ref{sym:F}) and (\ref{sym:F_more}) one can refer to \cite{Baker+Graves-Morris-Padeappr:96}. In the next section, we demonstrate that the model function approximates $F(\alpha)$ in the interval $[\alpha_1,\alpha_n]$.
\begin{rem}
The differentiation of $F^{(1)}(\alpha)$ yields $F^{(2)}(\alpha) = 2(Bx_\alpha^{(1)}, x_\alpha)$, and thus $ F^{(3)}(\alpha) =  2(Bx_\alpha^{(2)}, x_\alpha)+ 2(Bx_\alpha^{(1)}, x_\alpha^{(1)})$. Therefore it seems that the computation of $x_\alpha^{(2)}$ is required for the evaluation of $F^{(3)}(\alpha)$, however, from Theorem \ref{thm:F2k} it is enough to solve the equation for $x_{\alpha}^{(1)}$ for the evaluation of $F^{(3)}(\alpha)$ and is not necessary to compute $x_\alpha^{(2)}$.
\end{rem}
\section{Numerical tests}\label{numerics}
In this section, we present numerical tests to illustrate the efficiency of the proposed method using a linear inverse problems -- "{\bf heat}" 
which is generated from the Matlab package developed by Hansen \cite{Hansen-ReguToolvers:07}. 
$\varphi$ and $\psi$ are of the forms $\varphi(x,y^\delta)=\|Kx-y^\delta\|^2_{\R^{50}}$ with $50\times50$ matrix $K$ and $\psi(x)=\|x\|^2_{\R^{50}}$.

We use four conventional principles to determine the regularization parameters: the Morozov discrepancy principle, the damped Morozov principle, L-curve criterion and the minimum product criterion.
\\
{\bf The discrepancy principle} (\cite{Morozov-Methsolvincopose:84}) gives the regularization parameter $\alpha_{D}$ as the solution of the equation
\begin{equation}\label{morozov}
\varphi(x_\alpha,y^\delta)=\delta^2.
\end{equation}
{\bf The damped discrepancy principle} (\cite{Kunisch-clasdampMoroprin:93}) is a modification of the discrepancy principle that consists in determining $\alpha_{DP}$ such that
\begin{equation}\label{Dmorozov}
\varphi(x_\alpha,y^\delta) +  \alpha^\gamma\psi(x_\alpha)=\delta^2,
\end{equation}
In this test, we fix $\gamma=1$.\\
{\bf L-curve criterion} proposed in \cite{Hansen+OLeary-regudiscill-prob:93} gives the regularization parameter $\alpha_{L}$ which attains the largest magnitude of curvature $\kappa(\alpha)$ of the curve
   \[
   \{ (\log \psi(x_\alpha), \log \varphi(x_\alpha,y^\delta) ) \ \mid \ \alpha > 0 \}.
   \]
The curvature $\kappa(\alpha)$ can be written as \cite{Vogel-Non-reguparasele:96}
 \begin{equation}\label{L}
\kappa(\alpha)=\frac{\varphi\alpha\psi}{(\varphi^2 + \alpha^2\psi^2)^{\frac{3}{2}}}\left(\varphi+\alpha\psi+\frac{\varphi\psi}{\alpha\psi'}\right).
\end{equation}\\
{\bf The minimum product criterion }\cite{Reginska-reguparadiscill-:96} takes the regularization parameter $\alpha_{MP}$ as a local minimum point of the function
\begin{equation}\label{product}
\Psi_\gamma(\alpha):=\varphi(x_\alpha,y^\delta)^\gamma\psi(x_\alpha),\quad \gamma>0.
\end{equation}
In this test, we fix $\gamma=1$.\\
We rewrite all the principles in terms of $F$ by using the equations $\varphi(\alpha) = F(\alpha)-\alpha F'(\alpha)$, $\psi(\alpha) = F'(\alpha)$ and $\psi'(\alpha) = F''(\alpha)$ to employ model function approach.

Our purpose here is to identify the parameters $\alpha_{M},\alpha_{DM},\alpha_{L},\alpha_{MP}$ numerically with high accuracy by employing model function approach.
Note that we do not attempt to verify the effectiveness of those conventional principles for the inverse problem.
We also note that our approach is not restricted to these principles.
We can apply the approach to any principles that can be formulated in terms of the value function.
Although the size of the matrixes $K$ for the specific example is fixed to be $50\times 50$, our approach is available for much larger scale inverse problems.

The noisy data $y^\delta$ is generated by adding random noise to the exact data $y$ so that
\[ \varepsilon:=\frac{\|y-y^\delta\|_{\R^{50}}}{\|y\|_{\R^{50}}} \in \{0.01, 0.03, 0.05\}.\]

The regularization parameter $\alpha^{M}$ in Morozov principle 
is computed for each noise level as follows: we first compute $\varphi(x_{{\alpha}},y^\delta)$ in (\ref{morozov}) for 100 uniformly distributed $\alpha$-values in the interval $[10^{-8},10^{-0.5}]$ and find apprioximate optimals $\hat{\alpha}$ and $\bar{\alpha}$ that are very close, i.e, $|\hat{\alpha}-\bar{\alpha}|<<1$ and that satisfy $\varphi(x_{\hat{\alpha}},y^\delta)<\delta^2<\varphi(x_{\bar{\alpha}},y^\delta)$. Then a much smaller interval including these parameters is choosen to compute an accurate $\alpha^M$. The other parameters $\alpha^{DM},\alpha^{L},\alpha^{MP}$ in different principles are computed in a similar manner.

In each noise level, an approximation to the optimal parameter $\alpha^M$ in Morozov principle is computed by employing model functions as follows: Firstly, we construct the model function $m_1(\alpha)$ of the form
$\disp m_1(\alpha)=\Vert y^\delta\Vert_{\R^{50}}^2 \frac{\alpha^4 + p_{3}\alpha^{3}+p_{2}\alpha^{2}+p_1\alpha +p_0}{\alpha^4 + q_{3}\alpha^{3}+q_{2}\alpha^2 + q_1\alpha +q_0}$. The eight unknowns $\{p_k\}_{k=0}^{3}$, $\{q_k\}_{k=0}^{3}$ in this model function are determined by solving the linear system (\ref{sym:F}) where the value $F(\alpha_{i})$ and $F^{(1)}(\alpha_{i})$ are obtained at four points $\alpha_{i} \in \{ 10^{-8}, 10^{-5.5}, 10^{-3}, 10^{-0.5}\}$. Then we solve the equation $m_1(\alpha)-\alpha m_1'(\alpha) = \delta^2$, which is the model function version of $\varphi(x_{{\alpha}},y^\delta)=\delta^2$. We denote the solution by $\alpha^M_{1}$. We also construct another model function
$\disp m_{2}(\alpha) = \Vert y^\delta\Vert_{\R^{50}}^2 \frac{\alpha^8 + p_{7}\alpha^{7}+\cdots+p_1\alpha +p_0}{\alpha^8 + q_{7}\alpha^{7}+\cdots+ q_1\alpha +q_0}$ whose sixteen unknowns $\{p_k\}_{k=0}^{7}$, $\{q_k\}_{k=0}^{7}$ are determined by solving the linear system (\ref{sym:F}) where up to third derivatives $F(\alpha_{i})$, $F^{(1)}(\alpha_{i})$, $F^{(2)}(\alpha_{i})$ and $F^{(3)}(\alpha_{i})$ at the same points are used. We denote the solution of the equation $m_2(\alpha)-\alpha m_2'(\alpha) = \delta^2$ by $\alpha^M_2$. The approximated parameters in the other principles are computed in similar manner using the model function $m_1(\alpha)$ and $m_2(\alpha)$ and we denote them by $\alpha^{DM}_{1}$, $\alpha^{DM}_{2}$ (damped Morozov for $\lambda=1$), $\alpha^L_{1}$, $\alpha^L_{2}$ (L-curve) and $\alpha^{MP}_{1}$, $\alpha^{MP}_{2}$ (minimum product for $\gamma=1$), respectively.
All the computed parameters $\alpha^M$, $\alpha^M_1$, $\alpha^M_2$ etc in each noise level are reported in from Table \ref{Table1} to Table \ref{Table4}.
\begin{table}[h]
\begin{center}
\begin{tabular}{cc}
\begin{minipage}{0.5\hsize}
\caption{{\footnotesize$\alpha^M$, $\alpha^M_{1}$ and $\alpha^M_{2}$. Morozov.}}
\label{Table1}
  \begin{center}\footnotesize
    \begin{tabular}{c c c c }
    \hline
    $\varepsilon$    & 0.01                 & 0.03                & 0.05 \\
      $\alpha^M$     & $3.81\times10^{-5}$  & $1.39\times10^{-4}$ & $2.68\times10^{-4}$ \\
      $\alpha^M_{1}$ & $4.22\times10^{-5}$  & $5.97\times10^{-5}$ & $7.22\times10^{-5}$ \\
      $\alpha^M_{2}$ & $3.87\times10^{-5}$  & $1.25\times10^{-4}$ & $2.55\times10^{-4}$
    \end{tabular}
  \end{center}
\end{minipage}
\begin{minipage}{0.5\hsize}
  \caption{{\footnotesize$\alpha^{DM}$, $\alpha^{DM}_{1}$ and $\alpha^{DM}_{2}$. d-Morozov.}}
  \label{Table2}
  \begin{center}\footnotesize
    \begin{tabular}{c c c c }
    \hline
    $\varepsilon$       & 0.01                 & 0.03                & 0.05 \\
      $\alpha^{DM}$     & $1.85\times10^{-6}$  & $1.17\times10^{-5}$ & $2.84\times10^{-5}$ \\
      $\alpha^{DM}_{1}$ & $1.81\times10^{-6}$  & $1.00\times10^{-5}$ & $1.84\times10^{-5}$ \\
      $\alpha^{DM}_{2}$ & $1.85\times10^{-6}$  & $1.17\times10^{-5}$ & $2.74\times10^{-5}$
    \end{tabular}
  \end{center}
  \end{minipage}
  \end{tabular}
  \end{center}
\end{table}
\begin{table}[h]
\begin{center}
\begin{tabular}{cc}
\begin{minipage}{0.5\hsize}
  \caption{{\footnotesize$\alpha^{L}$, $\alpha^{L}_{1}$ and $\alpha^{L}_{2}$. L-curve.}}
  \label{Table3}
  \begin{center}\footnotesize
    \begin{tabular}{c c c c }
    \hline
    $\varepsilon$      & 0.01                 & 0.03                & 0.05 \\
      $\alpha^{L}$     & $2.80\times10^{-6}$  & $2.84\times10^{-5}$ & $7.22\times10^{-5}$ \\
      $\alpha^{L}_{1}$ & $3.05\times10^{-6}$  & $5.69\times10^{-6}$ & $3.86\times10^{-4}$ \\
      $\alpha^{L}_{2}$ & $2.52\times10^{-6}$  & $2.01\times10^{-5}$ & $9.36\times10^{-5}$
    \end{tabular}
  \end{center}
\end{minipage}
\begin{minipage}{0.5\hsize}
  \caption{{\footnotesize$\alpha^{MP}$, $\alpha^{MP}_{1}$ and $\alpha^{MP}_{2}$. Minimum product.}}
  \label{Table4}
  \begin{center}\footnotesize
    \begin{tabular}{c c c c }
    \hline
    $\varepsilon$       & 0.01                 & 0.03                & 0.05 \\
      $\alpha^{MP}$     & $8.50\times10^{-7}$  & $1.90\times10^{-5}$ & $7.22\times10^{-5}$ \\
      $\alpha^{MP}_{1}$ & $1.33\times10^{-6}$  & $9.72\times10^{-6}$ & $1.39\times10^{-4}$ \\
      $\alpha^{MP}_{2}$ & $8.79\times10^{-7}$  & $1.81\times10^{-5}$ & $7.47\times10^{-5}$
    \end{tabular}
  \end{center}
    \end{minipage}
  \end{tabular}
  \end{center}
\end{table}
The authors observed that the first derivative $m'_1(\alpha)$ did not approximate to $F'(\alpha)$.
On the other hand $m'_2(\alpha)$ was observed to give a very good approximation to $F'(\alpha)$.

The equation (\ref{Dmorozov}) is written in terms of $F$ as $F(\alpha) = \delta^2$. This means that the parameter $\alpha^{DM}$ is determined using only $F$. Thus it is enough to give a good approximation to $F$ to compute an approximation to $\alpha^{DM}$. As we expect, the parameters $\alpha^{DM}_1$ and $\alpha^{DM}_2$ in damped Morozov principle are very good approximations to $\alpha^{DM}$.

The parameters $\alpha^M_1$, $\alpha^L_1$ and $\alpha^{MP}_1$ for all nose level are not so accurate. This is because the equations (\ref{morozov}), (\ref{L}), (\ref{product}) contain the first derivative of $F$ and $m'_1(\alpha)$ does not approximate to $F'(\alpha)$.

To give a better approximation to $\alpha^M$, $\alpha^L$ and $\alpha^{MP}$, the model function must approximate to $F'$ in high accuracy and our model function $m_2(\alpha)$ will be the candidate.
Table \ref{Table1} and \ref{Table4} show that the parameters determined using $m_2(\alpha)$ are very accurate. On the other hand, $\alpha^L_2$ are not so accurate, although they are acceptably close to $\alpha^L$. An accurate second derivative of $F$ is also required for determining the parameter $\alpha^L$. 
Figure \ref{fig:L1_4points.eps} and \ref{fig:L3_4points.eps} show the curvature of the L-curve $\kappa(\alpha)$ with its numerical approximations $\kappa_{appro}(\alpha)$ obtained by using $m_{1}({\alpha})$ and $m_{2}(\alpha)$ respectively when the error $\varepsilon = 0.05$. The four reference points $\alpha_i\in\{ 10^{-8}, 10^{-5.5}, 10^{-3}, 10^{-0.5}\}$ used to construct $m_{1}({\alpha})$ and $m_{2}(\alpha)$ are depicted by bullets ($\bullet$) on the curve $\kappa(\alpha)$ in each Figure.
\begin{figure}[h]
\begin{minipage}{0.33\hsize}
\begin{center}
  \includegraphics[width=\hsize]{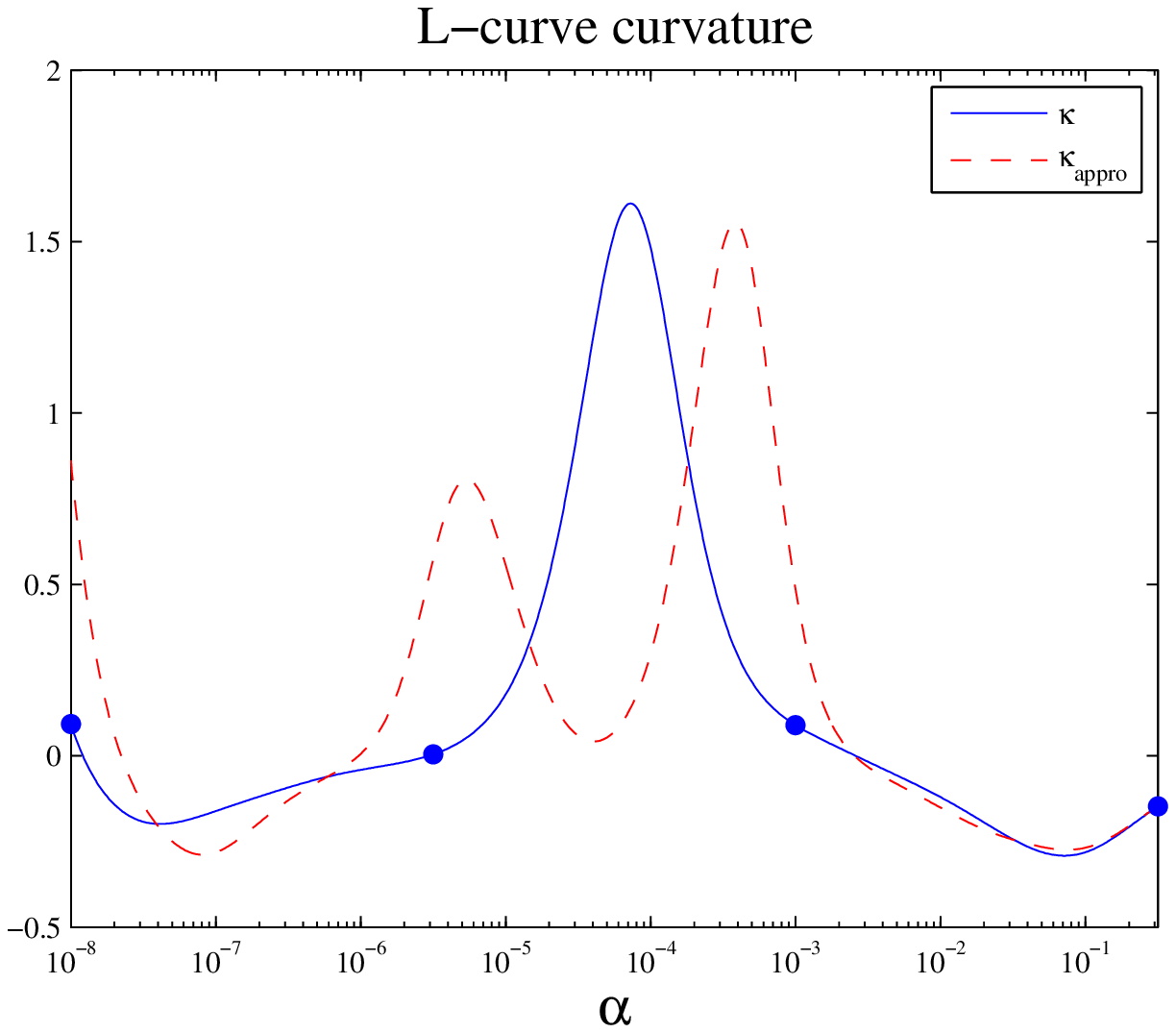}
  \end{center}
  \caption{\footnotesize $\kappa(\alpha)$ (solid line $-$) and $\kappa_{appro}(\alpha)$ (dashed line $- -$) obtained by using $m_{1}(\alpha)$ with four reference points $\alpha_i$ $(\bullet)$.} \label{fig:L1_4points.eps}
  \end{minipage}
  \begin{minipage}{0.33\hsize}
\begin{center}
  \includegraphics[width=\hsize]{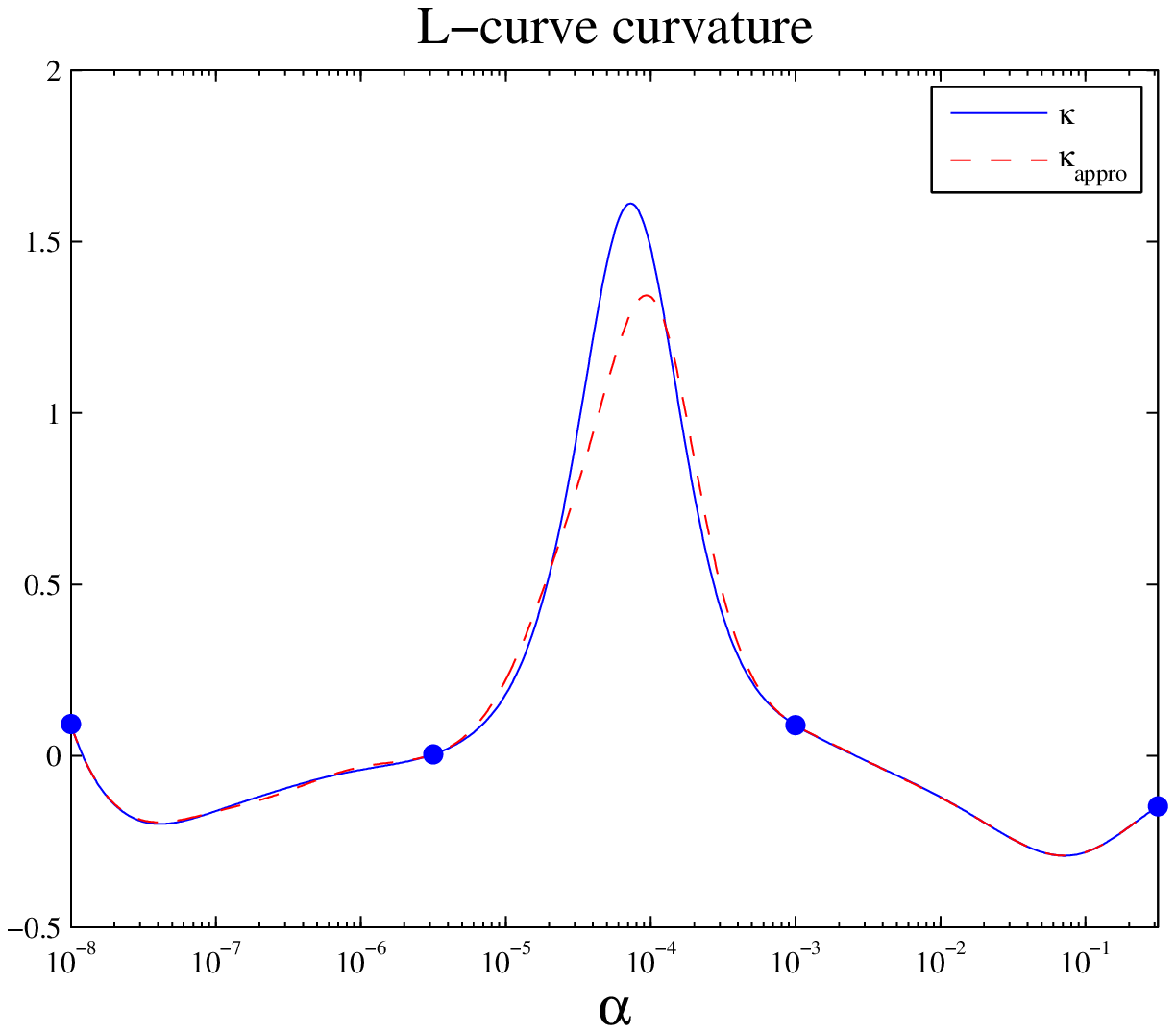}
    \end{center}
   \caption{\footnotesize $\kappa(\alpha)$ (solid line $-$) and $\kappa_{appro}(\alpha)$ (dashed line $- -$) obtained by using $m_{2}(\alpha)$ with four reference points $\alpha_i$ $(\bullet)$.}  \label{fig:L3_4points.eps}
  \end{minipage}
  \begin{minipage}{0.33\hsize}
  \begin{center}
    \includegraphics[width=\hsize]{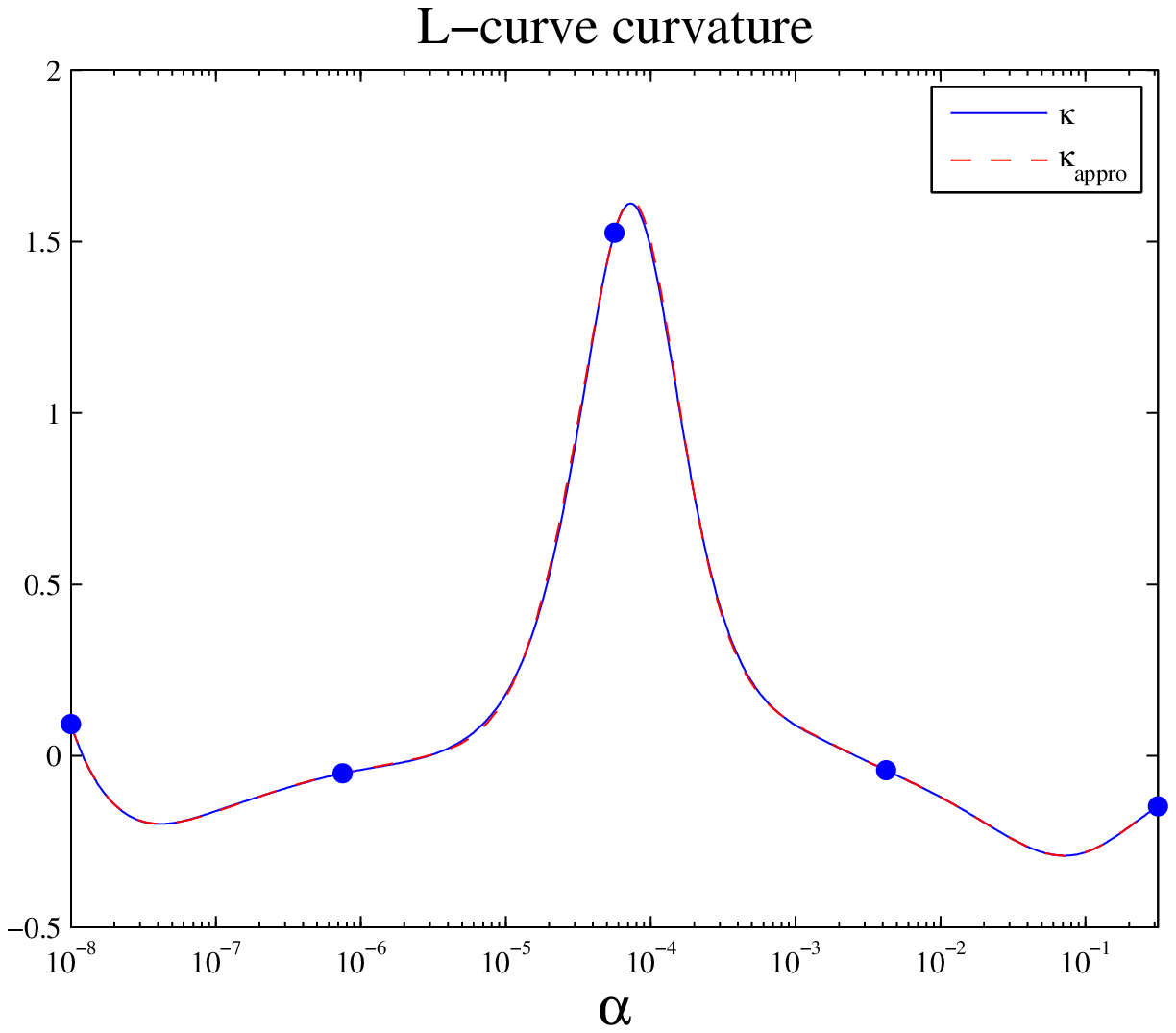}
    \end{center}
   \caption{\footnotesize $\kappa(\alpha)$ (solid line $-$) and $\kappa_{appro}(\alpha)$ (dashed line $- -$) obtained by using $m_{3}(\alpha)$ with five reference points $\alpha_i$ $(\bullet)$.}  \label{fig:L3_5points.eps}
   \end{minipage}
\end{figure}
The approximation $\kappa_{appro}(\alpha)$ ( dashed line $- -$ ) in Figure \ref{fig:L1_4points.eps} completely fails to approximate to $\kappa(\alpha)$.
We used $F(\alpha_{i})$ and $F^{(1)}(\alpha_{i})$ for the construction of the model function $m_{1}(\alpha)$, and thus the second derivative $m_{1}^{(2)}(\alpha)$ can not approximate to $F^{(2)}(\alpha)$ which is contained in $\kappa(\alpha)$. On the other hand, our $m_2(\alpha)$ gives better approximation to $\kappa$ as shown in Figure \ref{fig:L3_4points.eps}, although $\kappa_{appro}(\alpha)$ does not much perfectly with $\kappa$. To give more accurate $\kappa_{appro}(\alpha)$, we construct another model function of the form
$\disp m_3(\alpha) =\Vert y^\delta\Vert_{\R^{50}}^2 \frac{\alpha^{10}+ p_{9}\alpha^{9}+\cdots+p_1\alpha +p_0}{\alpha^{10} + q_{9}\alpha^{9}+\cdots+ q_1\alpha +q_0}$ using $F(\alpha_{i})$, $F^{(1)}(\alpha_{i})$, $F^{(2)}(\alpha_{i})$ and $F^{(3)}(\alpha_{i})$ at five reference points $\alpha_{i}\in\{10^{-8}, 10^{ -6.125}, 10^{-4.25}, 10^{-2.375}, 10^{-0.5}\}$.

Figure \ref{fig:L3_5points.eps} depicts the curvature of the L-curve $\kappa(\alpha)$ with its numerical approximations $\kappa_{appro}(\alpha)$ obtained by using $m_{3}(\alpha)$. The five reference points are also shown in the Figure.
We observe that the model function $m_{3}(\alpha)$ yields the sufficiently good approximation $\kappa_{appro}(\alpha)$ that almost perfectly matches with the exact curvature $\kappa(\alpha)$. This observation suggests that we should use more reference points to construct model function when we employ a principle that contains second derivatives of the value function.

\begin{rem}
In a practical situation, an interval where a regularization parameter to be found is often much smaller than the interval $(10^{-8}, 10^{-0.5})$ used for our numerical test. 
If it is the case it is enough to construct a model function in a smaller interval.
The number of reference points to be used for the construction of a model function can be reduced to two or three. 
 \end{rem}

\section{A new criterion of the choice of the regularization parameter}\label{sec:ParaF}
We propose a new criterion for the regularization parameter.
Let us introduce a function $\Gamma_\gamma(\alpha)$ defined as
\begin{equation}\label{NewF}
\Gamma_\gamma(\alpha)
=\frac{\gamma^\gamma}{(1+\gamma)^{1+\gamma}}\frac{F^{1+\gamma}(\alpha)}{\alpha},
\end{equation}
where $\gamma$ is a positive constant.

Our new criterion takes the parameter as
a local minimum $\alpha$ of $ \Gamma_\gamma(\alpha)$.
Since
$\displaystyle
\frac{d \Gamma_\gamma(\alpha)}{d\alpha} =\frac{\gamma^\gamma}{(1+\gamma)^{1+\gamma}} F^\gamma(\alpha)((1+\gamma)\alpha F'(\alpha)-F(\alpha))\alpha^{-2},
$
and $F(\alpha)>0$ for all $\alpha>0$, $\alpha_\gamma$ solves the equation $(1+\gamma)\alpha F'(\alpha)-F(\alpha)=0$.

The criterion is similar to the minimum product criterion by Regi\'{n}ska \cite{Reginska-reguparadiscill-:96}.
First, we note that the energy function $\Psi_\gamma(\alpha)$ in (\ref{product}) is written in terms of $F(\alpha)$
\begin{equation}\label{MinimumF}
\Psi_\gamma(\alpha)=(F(\alpha)-\alpha F'(\alpha))^\gamma F'(\alpha).
\end{equation}
Suppose that $F''(\alpha)$ exists and $F''(\alpha)<0$. (A sufficient condition for the existence of the second derivative and the negativity can be found in \cite{Ito+Kunisch-choireguparanonl:92}.)
Since
$\ds\frac{d}{d\alpha}\Psi_\gamma(\alpha)=(F(\alpha)-\alpha F'(\alpha))^{\gamma-1}F''(\alpha)(F(\alpha)-(1+\gamma)\alpha F'(\alpha))=\varphi(x_\alpha,y^\delta)^{\gamma-1}F''(\alpha)(F(\alpha)-(1+\gamma)\alpha F'(\alpha))$, the regularization parameter determined by the criterion solves the
equation $(1+\gamma)\alpha F'(\alpha)-F(\alpha)=0$.

The relationship between (\ref{NewF}) and (\ref{MinimumF}) follows from the next Proposition.
\begin{Prop}
Let $\gamma>0$ be a positive number.
\begin{equation}
\Psi_\gamma(\alpha)
\le\Gamma_\gamma(\alpha), \quad \text{ for all } \alpha>0.
\end{equation}
The equality holds if and only if $\alpha$ solves the equation
$\varphi(\alpha) - \gamma\alpha\psi(\alpha)=0$.
\end{Prop}
\begin{proof}
Consider the inequality $ab\le a^p/p+b^q/q$ with
 $p^{-1}=\gamma(1+\gamma)^{-1}$ and $q^{-1} = (1+\gamma)^{-1}$.
Substituting $ {\varphi^{\frac{\gamma}{1+\gamma}}}{\alpha^{-\frac{\gamma}{2(1+\gamma)}}}$ with $a$
 and $(\gamma\psi)^{\frac{1}{1+\gamma}}\alpha^{\frac{1}{2(1+\gamma)}}$ with $b$,
it follows that
\[
\varphi^{\frac{\gamma}{1+\gamma}}(\gamma\psi)^{\frac{1}{1+\gamma}}\alpha^{\frac{1-\gamma}{2(1+\gamma)}}
\le
\frac{\gamma}{1+\gamma}\frac{\varphi + \alpha\psi}{\alpha^{\frac{1}{2}}}
=
\frac{\gamma}{1+\gamma}\frac{F(\alpha)}{\alpha^{\frac{1}{2}}}
\]
and the inequality holds if and only if $a^p = b^p$, namely,
$
\varphi(\alpha) - \gamma\alpha\psi(\alpha)=0.
$
Multiplying $\alpha^{-\frac{1-\gamma}{2(1+\gamma)}}$ and taking $1+\gamma$ power yields the desired inequality. 
\end{proof}
Figure \ref{fig:Gamma.eps} shows  $\Gamma_1(\alpha)$ and $\Psi_1(\alpha)$ in the interval $(10^{-10},10^{-1})$ for certain linear inverse problem.
\begin{figure}[h]
\begin{center}
 \includegraphics[width=.5\hsize]{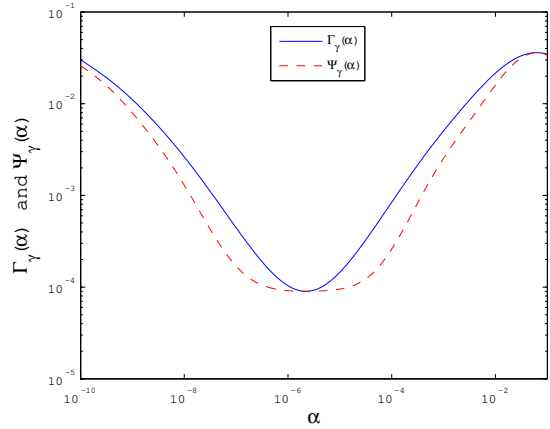}\\
    \caption{The graph of $\Gamma_1(\alpha)$ (solid line $-$) and $\Psi_1(\alpha)$ (dashed line $- -$)}\label{fig:Gamma.eps}
    \end{center}
\end{figure}
There exists a local minimum point around $\alpha=10^{-5}$ where $\Gamma_1(\alpha)$ and $\Psi_1(\alpha)$ take the same value.

The advantages of our criterion are (i) the shape of $\Gamma_\gamma(\alpha)$ is sharper than $\Psi_\gamma(\alpha)$ and thus it is easier to detect the minimum point. (ii) $\Gamma_\gamma(\alpha)$ contains $F(\alpha)$ only, does not require the function $F'(\alpha)$ which can be discontinuous due to the nonuniqueness of an inverse problem.

The effect of the parameter $\gamma$ in $\Gamma_{\gamma}$ to the quality of the solution $x_\alpha$ should be studied. We investigate both the applicability of the criterion to nonlinear problems and the effect of the parameter $\gamma$ in our future works.

\section{Conclusion}
We investigate the minimum value function for the Tikhonov regularization. We propose the model function for the minimum value function for linear inverse problems and verify its efficiency in the determination of the regularization parameter. We also propose a new criterion for the choice of the regularization parameter. Our criterion strongly relates to the minimum product criterion and is applicable to nonlinear inverse problems.

\bibliographystyle{plain} 
\def\cprime{$'$} \def\cprime{$'$} \def\cprime{$'$}
  \def\cydot{\leavevmode\raise.4ex\hbox{.}}
  \def\Dbar{\leavevmode\lower.6ex\hbox to 0pt{\hskip-.23ex \accent"16\hss}D}
  \def\cprime{$'$}

\end{document}